\documentclass[11pt]{amsart}
\usepackage{graphicx}
\usepackage{amssymb, amsmath}
\newtheorem{theorem}{Theorem}[section]
\newtheorem{corollary}[theorem]{Corollary}
\newtheorem{lemma}[theorem]{Lemma}
\newtheorem{proposition}[theorem]{Proposition}
\theoremstyle{definition}
\newtheorem{definition}[theorem]{Definition}
\newtheorem{example}[theorem]{Example}
\newtheorem{remark}[theorem]{Remark}
\theoremstyle{remark}

\numberwithin{equation}{section}

\begin{document}
\title{Representation Theory of Polyadic  Groups}
\author{W. A. Dudek}
\address{Institute of Mathematics and Computer Science,
Wroc{\l}aw University of Techno\-logy, Wybrze\.ze Wyspia\'nskiego
27, 50-370 Wroc{\l}aw, Poland} \email{dudek@im.pwr.wroc.pl}
\author{\sc M. Shahryari}
\thanks{{\scriptsize
\hskip -0.4 true cm MSC(2010): 20N15
}}

\address{ Department of Pure Mathematics,  Faculty of Mathematical
Sciences, University of Tabriz, Tabriz, Iran }
\email{mshahryari@tabrizu.ac.ir}
\date{\today}

%%% ----------------------------------------------------------------------
\begin{abstract}
In this article, we introduce the notion of representations of
polyadic groups and we investigate the connection between these
representations and those of retract groups and covering groups.
\end{abstract}

\maketitle

%%%%%%%%%%%%%%%%%%%%%%%%%%%%%%%%%%%%%%%%%%%%%%%%%%%%%%%%%%%%%%%%%%%%%%

\section{Introduction}
A non-empty set $G$ together with an $n$-ary operation $f:G^n\to
G$ is called an {\it $n$-ary groupoid} and is denoted by $(G,f)$.
We will assume that $n>2$.

According to the general convention used in the theory of $n$-ary systems,
the sequence of elements $x_i,x_{i+1},\ldots,x_j$ is denoted by
$x_i^j$. In the case $j<i$ it is the empty symbol. If
$x_{i+1}=x_{i+2}=\ldots=x_{i+t}=x$, then instead of
$x_{i+1}^{i+t}$ we write $\stackrel{(t)}{x}$. In this convention
$f(x_1,\ldots,x_n)= f(x_1^n)$ and
 \[
 f(x_1,\ldots,x_i,\underbrace{x,\ldots,x}_{t},x_{i+t+1},\ldots,x_n)=
 f(x_1^i,\stackrel{(t)}{x},x_{i+t+1}^n) .
 \]

An $n$-ary groupoid $(G,f)$ is called {\it $(i,j)$-associative}, if
 \begin{equation}
f(x_1^{i-1},f(x_i^{n+i-1}),x_{n+i}^{2n-1})=
f(x_1^{j-1},f(x_j^{n+j-1}),x_{n+j}^{2n-1})\label{assolaw}
 \end{equation}
holds for all $x_1,\ldots,x_{2n-1}\in G$. If this identity holds
for all $1\leqslant i<j\leqslant n$, then we say that the
operation $f$ is {\it associative} and $(G,f)$ is called an {\it
$n$-ary semigroup}.

If, for all $x_0,x_1,\ldots,x_n\in G$ and fixed
$i\in\{1,\ldots,n\}$, there exists an element $z\in G$ such that
\begin{equation}\label{solv}
f(x_1^{i-1},z,x_{i+1}^n)=x_0 ,
\end{equation}
then we say that this equation is {\it $i$-solvable} or {\it
solvable at the place $i$}. If this solution is unique, then we
say that (\ref{solv}) is {\it uniquely $i$-solvable}.

An $n$-ary groupoid $(G,f)$ uniquely solvable for all
$i=1,\ldots,n$, is called an {\it $n$-ary quasigroup}. An
associative $n$-ary quasigroup is called an {\it $n$-ary group} or
a {\it polyadic group}. In the binary case (i.e., for $n=2$) it is
a usual group.

Now, such and similar $n$-ary systems have many applications in
different branches. For example, in the theory of automata, (cf.
\cite{Busse}), $n$-ary semigroups and $n$-ary groups are used,
some $n$-ary groupoids are applied in the theory of quantum groups
(cf. \cite{Nik}). Different applications of ternary structures in
physics are described by R. Kerner (cf. \cite{Ker}). In physics
there are used also such structures as $n$-ary Filippov algebras
(cf. \cite{Poj}) and $n$-Lie algebras (cf. \cite{Vai}).

The idea of investigations of such groups seems to be going back
to E. Kasner's lecture \cite{Kas} at the fifty-third annual
meeting of the American Association for the Advancement of Science
in 1904. But the first paper concerning the theory of $n$-ary
groups was written (under inspiration of Emmy Noether) by W.
D\"ornte in 1928 (see \cite{Dor}). In this paper D\"ornte observed
that any $n$-ary groupoid $(G,f)$ of the form $\,f(x_1^n)=x_1\circ
x_2\circ\ldots\circ x_n\circ b$, where $(G,\circ)$ is a group and
$b$ is its fixed element belonging to the center of $(G,\circ)$,
is an $n$-ary group. Such $n$-ary groups, called {\it $b$-derived}
from the group $(G,\circ)$, are denoted by $der_b(G,\circ)$. In
the case when $b$ is the identity of $(G,\circ)$ we say that such
$n$-ary group is {\it reducible} to the group $(G,\circ )$ or {\it
derived} from $(G,\circ)$. But for every $n>2$ there are $n$-ary
groups which are not derived from any group. An $n$-ary group
$(G,f)$ is derived from some group iff it contains an element $e$
(called an {\it $n$-ary identity}) such that
 \begin{equation}\label{n-id}
 f(\stackrel{(i-1)}{e},x,\stackrel{(n-i)}{e})=x
 \end{equation}
holds for all $x\in G$ and $i=1,\ldots,n$.

It is worthwhile to note that in the definition of an $n$-ary
group, under the assumption of the associativity of the operation
$f$, it suffices only to postulate the existence of a solution of
(\ref{solv}) at the places $i=1$ and $i=n$ or at one place $i$
other than $1$ and $n$ (cf. \cite{Post}, p. $213$). Other
useful characterizations of $n$-ary groups one can find in
\cite{Rem} and \cite{DGG}.

\medskip

From the definition of an $n$-ary group $(G,f)$, we can directly
see that for every $x\in G$, there exists only one $z\in G$
satisfying the equation
 \begin{equation} \label{skew}
f(\stackrel{(n-1)}{x},z)=x .
 \end{equation}
This element is called {\it skew} to $x$ and is denoted by
$\overline{x}$. In a ternary group ($n=3$) derived from the binary
group $(G,\cdot)$ the skew element coincides with the inverse
element in $(G,\circ)$. Thus, in some sense, the skew element is a
generalization of the inverse element in binary groups.  D\"ornte
proved (see \cite{Dor}) that in ternary groups we have
$\overline{f(x,y,z)}=f(\overline{z},\overline{y},\overline{x})$
and $\overline{\overline{x}}=x$, but for $n>3$ this is not true.
For $n>3$ there are $n$-ary groups in which one fixed element is
skew to all elements (cf. \cite{D'90}) and $n$-ary groups in which
any element is skew to itself.

Nevertheless, the concept of skew elements plays a crucial role in
the theory of $n$-ary groups. Namely, as D\"ornte proved (see also
\cite{DGG}), the following theorem is true.

\begin{theorem}\label{dor-th}
In any $n$-ary group $(G,f)$ the following identities
\begin{equation}\label{e-dor}
f(\stackrel{(i-2)}{x},\overline{x},\stackrel{(n-i)}{x},y)=
f(y,\stackrel{(n-j)}{x},\overline{x},\stackrel{(j-2)}{x})=y,
\end{equation}
\begin{equation}\label{e-skew}
f(\stackrel{(k-1)}{x},\overline{x},\stackrel{(n-k)}{x})=x
\end{equation}
 hold for all $\,x,y\in G$, $\,2\leqslant i,j\leqslant
n$ and $1\leqslant k\leqslant n$.
\end{theorem}

One can prove (cf. \cite{Rem}) that for $n>2$ an $n$-ary group can
be defined as an algebra $(G,f,\bar{\,}\,)$ with one associative
$n$-ary operation $f$ and one unary operation
$\bar{\,}:x\to\overline{x}$ satisfying for some $\,2\leqslant
i,j\leqslant n$ the identities \eqref{e-dor}. This means that a
non-empty subset $H$ of an $n$-ary group $(G,f)$ is its subgroup
iff it is closed with respect to the operation $f$ and
$\overline{x}\in H$ for every $x\in H$.

Fixing in an $n$-ary operation $f$ all inner elements
$a_2,\ldots,a_{n-1}$ we obtain a new binary operation
\[
x\ast y=f(x,a_2^{n-1},y).
 \]
Such obtained groupoid $(G,\ast)$ is called a {\it retract} of
$(G,f)$. Choosing different elements $a_1,\ldots,a_{n-1}$ we
obtain different retracts. Retracts of $n$-ary groups are groups.
Retracts of a fixed $n$-ary group are isomorphic (cf.
\cite{DM'84}). So, we can consider only retracts of the form
\[
x\ast y=f(x,\stackrel{(n-2)}{a},y).
\]
Such retracts will be denoted by $Ret_a(G,f)$, or simply by $Ret_a(G)$. The identity of the
group $Ret_a(G)$ is $\overline{a}$. One can verify that the
inverse element to $x$ has the form
\begin{equation}\label{inverse}
x^{-1}=f(\overline{a},\stackrel{(n-3)}{x},\overline{x},\overline{a}).
\end{equation}

Binary retracts of an $n$-ary group $(G,f)$ are commutative only
in the case when there exists an element $a\in G$ such that
 \[
f(x,\stackrel{(n-2)}{a},y)=f(y,\stackrel{(n-2)}{a},x)
 \]
holds for all $x,y\in G$. An $n$-ary group with this property is
called {\em semiabelian}. It satisfies the identity
\begin{equation}\label{semi}
 f(x_1^n)=f(x_n,x_2^{n-1},x_1)
\end{equation}
(cf. \cite{Rem}).

One can prove (cf. \cite{GG'77}) that a semiabelian $n$-ary group
is {\it medial}, i.e., it satisfies the identity
\begin{equation}\label{med-law}
f(f(x_{11}^{1n}),f(x_{21}^{2n}),\ldots,f(x_{n1}^{nn}))=f(f(x_{11}^{n1}),f(x_{12}^{n2}),\ldots,
f(x_{1n}^{nn})).
\end{equation}
In such $n$-ary groups
\begin{equation}\label{e-med}
\overline{f(x_1,x_2,x_3,\ldots,x_n)}=f(\overline{x}_1,\overline{x}_2,\overline{x}_3,\ldots,\overline{x}_n)
\end{equation}
for all $x_1,\ldots,x_n\in G$.

Any $n$-ary group can be uniquely described by its retract and
some automorphism of this retract. Namely, the following
Hossz\'u-Gluskin Theorem (cf. \cite{DG} or \cite{DM'82}) is valid.

\begin{theorem}\label{T-HG}
An $n$-ary groupoid $(G,f)$ is an $n$-ary group iff

\begin{enumerate}
\item[$(1)$] \ on $G$ one can define an operation $\cdot$ such
that $(G,\cdot)$ is a group,
\item[$(2)$] \ there exist an automorphism $\varphi$ of $(G,\cdot)$ and $b\in G$ such
that $\varphi(b)=b$,
\item[$(3)$] \ $\varphi^{n-1}(x)=b\cdot x\cdot b^{-1}$ for every $x\in
G$,
\item[$(4)$] \ $
f(x_1^n)=x_1\cdot\varphi(x_2)\cdot\varphi^2(x_3)\cdot\cdots\cdot\varphi^{n-1}(x_n)\cdot
b$ for all $x_1,\ldots,x_n\in G$.
\end{enumerate}
\end{theorem}

One can prove that $(G,\cdot)= Ret_a(G,f)$ for some $a\in G$. In
connection with this we say that an $n$-ary group $(G,f)$ is
$(\varphi,b)$-derived from the group $(G,\cdot)$.

The main aim of this article is to introduce {\em representations}
of $n$-ary groups and to investigate their main properties, with a
special focus on ternary groups. Note that, this is not the first
attempt to study representations of $n$-ary groups, because there
are some other articles, with different point of views concerning
representations on $n$-ary groups, (cf. \cite{BDD}, \cite{GWW},
\cite{Post} and \cite{Wanke}). However, our method seems to be the
most natural generalization of the notion of representation from
binary to $n$-ary groups.

\section{Action of an $n$-ary Group on a Set}

Suppose that $(G,f)$ is an $n$-ary group and $A$ is a non-empty
set. We say that $(G,f)$ {\it acts} on $A$ if for all $x\in G$ and
$a\in A$ corresponds a unique element $x.a\in A$ such that
\begin{enumerate}
\item[$(i)$]\ \ $f(x_1^n).a=x_1.(x_2.(x_3.\ \ldots\ .(x_n.a))\ldots )$ for all $x_1,\ldots,x_n\in G$,

\item[$(ii)$]\ \ for all $a\in A$, there exists $x\in G$ such that
$x.a=a$,

\item[$(iii)$]\ \ the map $a\mapsto x.a$ is a bijection for all $x\in
G$.
\end{enumerate}

\medskip

For $a\in A$, we define the {\it stabilizer } $G_a$ of $a$ as
follows
$$
G_a=\{ x\in G: x.a=a\}.
$$
\begin{proposition}
$G_a$ is an $n$-ary subgroup of $(G,f)$.
\end{proposition}
\begin{proof} By condition $(ii)$ of the above definition $G_a$ is non-empty.
Since for $x_1, x_2, \ldots, x_n\in G_a$ we have
\[
f(x_1^n).a=x_1.(x_2.(x_3.\ \ldots\ .(x_n.a))\ldots )=a,
\]
$f(x_1^n)\in G_a$. Hence $G_a$ is closed with respect to the
operation $f$.

Now if $x\in G_a$, then by \eqref{e-skew} we obtain
$$
a=x.a=f(\overline{x},\stackrel{(n-1)}{x}).a=\overline{x}.(x.\,\ldots\,.x.(x.a))\ldots)=\overline{x}.a
,
$$
which implies $\overline{x}\in G_a$. This completes the proof.
\end{proof}

\begin{proposition}\label{P-equiv}
If an $n$-ary group $(G,f)$ acts on a set $A$, then the relation
$\sim$ defined on $A$ by
$$
a\sim b \Longleftrightarrow \exists x\in G:\ x.a=b
$$
is an equivalence relation.
\end{proposition}
\begin{proof}
For each $a\in A$ there is $x\in G$ such that $x.a=a$, so $a\sim
a$. If $a\sim b$ for $a,b\in A$, then $z.a=b$ for some $z\in G$.
Let $y$ be the unique solution of the equation
$$
f(y,z,\stackrel{(n-2)}{x})=x ,
$$
where $x\in G$ is such that $x.a=a$. For this $y$ we have $y.b=a$
since
\[
a=x.a=f(y,z,\stackrel{(n-2)}{x}).a=y.z.a=y.b .
\]
Thus $b\sim a$. Finally, let $a\sim b$ and $b\sim c$. Then there
are $x,y,z\in G$ such that $x.a=b$, $y.b=c$ and $z.b=b$. In this
case for $u=f(y,\stackrel{(n-2)}{z},x)$ we have
\[
u.a=f(y,\stackrel{(n-2)}{z},x).a=y.b=c,
\]
which proves $a\sim c$.
\end{proof}

\begin{theorem}\label{T-act}
The formula $x.a=f(x,a,\stackrel{(n-3)}{x},\overline{x})$ defines
an action of an $n$-ary group $G$ on itself.
\end{theorem}
\begin{proof}
The last condition of Theorem \ref{T-HG} can be written in the
form
\[
f(x_1^n)=x_1\cdot\varphi(x_2)\cdot\varphi^2(x_3)\cdot\ldots\cdot\varphi^{n-2}(x_{n-1})\cdot
b\cdot x_n .
\]
Thus
$\,\overline{x}=\big(\varphi(x)\cdot\varphi^2(x)\cdot\ldots\cdot\varphi^{n-2}(x)\cdot
b\big)^{-1}$. Consequently
\begin{equation}\label{e-n}
x.a=x\cdot\varphi(a)\cdot\varphi(x^{-1}) .
\end{equation}
Hence
\begin{eqnarray*}
y.(x.a)&=&y\cdot \varphi(x)\cdot\varphi^2(a)\cdot\varphi^2(x^{-1})\cdot\varphi(y)^{-1}\\
       &=&y\cdot \varphi(x)\cdot\varphi^2(a)\cdot\varphi((y\cdot\varphi(x))^{-1}) .
\end{eqnarray*}
Iterating this procedure we obtain
$$
x_1.(x_2.(x_3\,\ldots\,.(x_n.a))\ldots)=
$$
$$
x_1\cdot\varphi(x_2)\cdot\varphi^2(x_3)\cdot\ldots\cdot\varphi^{n-1}(x_n)\cdot\varphi^n(a)\cdot\varphi((x_1\cdot\varphi(x_2)\cdot
\varphi^2(x_3)\cdot\ldots\cdot\varphi^{n-1}(x_n))^{-1}) .
$$
Since $\varphi^n(a)=b\cdot\varphi(a)\cdot b^{-1}$ from the above
we obtain
$$
x_1.(x_2.(x_3\,\ldots\,.(x_n.a))\ldots)=f(x_1^n)\cdot\varphi(a)\cdot\varphi(f(x_1^n)^{-1})
.
$$
This by \eqref{e-n} gives
$f(x_1^n).a=x_1.(x_2.(x_3\,\ldots\,.(x_n.a))\ldots)$.
\end{proof}

\begin{proposition}
In semiabelian $n$-ary groups the relation
$$
a\sim b \Longleftrightarrow \exists x\in G:\
f(x,a,\stackrel{(n-3)}{x},\overline{x})=b
$$
is a congruence.
\end{proposition}
\begin{proof}
Indeed, by Proposition \ref{P-equiv} it is an equivalence
relation. To prove that it is a congruence let $a_i\sim b_i$,
i.e., $f(x_i,a_i,\stackrel{(n-3)}{x_i},\overline{x}_i)=b_i$ for
some $x_i\in G$ and all $i=1,\ldots,n$. Then
\[
f(b_1^n)=f(f(x_1,a_1,\stackrel{(n-3)}{x_1},\overline{x}_1),f(x_2,a_2,\stackrel{(n-3)}{x_2},\overline{x}_2),
\ldots,f(x_n,a_n,\stackrel{(n-3)}{x_n},\overline{x}_n)),
\]
which by the mediality and \eqref{e-med} gives
\[
f(b_1^n)=f(f(x_1^n),f(a_1^n),\underbrace{f(x_1^n),\ldots,f(x_1^n)}_{n-3},\overline{f(x_1^n)}\,).
\]
Thus $f(a_1^n)\sim f(b_1^n)$.
\end{proof}

\begin{remark}
The formula \eqref{e-n} says that in $n$-ary groups $b$-derived
from a group $(G,\cdot)$ the above relation coincides with the
conjugation in $(G,\cdot)$. Thus in non-semiabelian $n$-ary groups
it may not be a congruence.
\end{remark}

Elements belonging to the same equivalence class are called {\it
conjugate}. The equivalence classes are called {\it conjugate
classes} of an $n$-ary group $G$ and have the form
$$
Cl_G(a)=\{ f(x,a,\stackrel{(n-3)}{x},\overline{x}): x\in G\}.
$$

As a simple consequence of \eqref{med-law} and \eqref{e-med} we
obtain
\begin{proposition}
In semiabelian $n$-ary group the set containing all elements of
$G$ conjugated with elements of a given $n$-ary subgroup also is
an $n$-ary subgroup.
\end{proposition}

For $a\in G$, we define the {\it centralizer} of $a$, as follows
$$
C_G(a)=\{ x\in G: f(x,a,\stackrel{(n-3)}{x},\overline{x})=a\}.
$$
From Theorem \ref{dor-th} it follows that in $n$-ary groups
$b$-derived from a group $(G,\cdot)$ the centralizer of any $a\in
G$ coincides with the centralizer of $a$ in $(G,\cdot)$.

\begin{proposition}\label{P-23}
For every $x\in C_G(a)$ and every $0\leqslant i,j,k\leqslant n-2$
such that $i+j+k=n-2$ we have
$$
f(\stackrel{(i)}{x},a,\stackrel{(j)}{x},\overline{x},\stackrel{(k)}{x})=
f(\stackrel{(i)}{x},\overline{x},\stackrel{(j)}{x},a,\stackrel{(k)}{x})=a.
$$
\end{proposition}
\begin{proof}
For every $x\in C_G(a)$, we have
$f(x,a,\stackrel{(n-3)}{x},\overline{x})=a$. Multiplying this
equation on the left by $x$ and on the right by
$x,\ldots,x,\overline{x}$ ($n-2$ elements), we obtain
$$
f(x,f(x,a,\stackrel{(n-3)}{x},\overline{x}),\stackrel{(n-3)}{x},\overline{x})=f(x,a,\stackrel{(n-3)}{x},\overline{x})=a,
$$
which in view of the associativity of the operation $f$ and
\eqref{e-skew} gives
$$
f(x,x,a,\stackrel{(n-4)}{x},\overline{x})=a.
$$
Repeating this procedure we obtain
$$
f(\stackrel{(i)}{x},a,\stackrel{(n-i-2)}{x},\overline{x})=a
$$
for every $1\leqslant i\leqslant n-2$. Theorem \ref{dor-th}
completes the proof.
\end{proof}

\section{ G-modules and Representations}

All vector spaces in this section  are defined over the field of
complex numbers and have  finite dimension.
\begin{definition}
Suppose that an $n$-ary group $G$ acts on a vector space $V$ and
we have
\begin{enumerate}
\item \ $x.(\lambda v+u)=\lambda x.v+x.u$,
\item \ $\exists p\in G\ \forall v\in V: p.v=v$.
\end{enumerate}
Then we call $(V, p)$, or simply $V,$ a {\it $G$-module}.
\end{definition}

Notions, such as $G$-submodule, $G$-homomorphism, irreducibility
and so on, are defined by the ordinary way.

\begin{definition}
A map $\Lambda:G \rightarrow GL(V)$ with the property
$$
\Lambda(f(x_1^n))=\Lambda(x_1)\Lambda(x_2)\ldots\Lambda(x_n)
$$
is a {\it representation} of $G$, provided that $\ker \Lambda$ is
non-empty. The function
$$
\chi(x)=Tr\ \Lambda(x)
$$
is called the corresponding {\it character} of $\Lambda$.
\end{definition}

\begin{remark}
If $V$ is a $G$-module, then $\Lambda$ defined by
$$
\Lambda(x)(v)=x.v
$$
is a representation of $G$. The converse is also true.
\end{remark}

\begin{example}\label{Ex-34}
Let $A$ be an arbitrary binary group with a normal subgroup $H.$
Let $a\in A\setminus H$ be an involution. Then $G=aH$ with the
operation
$$
f(x,y,z)=xyz
$$
is a ternary group. If $\Lambda$ is an ordinary representation of
$A$ with the property $a\in\ker\Lambda$, then, clearly $\Lambda$
is also a representation of $G$. For example, suppose
$A=GL_n(\mathbb{C})$ and $H=SL_n(\mathbb{C})$. Let $a=diag(-1, 1,
\ldots, 1)$ and define $G=aH$. Then, every representation of $A$
in which $a\in \ker \Lambda$ is also a representation of a ternary
group $G$.
\end{example}

\begin{example}\label{ex-35}
For any subgroup $H$ of an ordinary group $A$ and any element
$a\in Z(A)\setminus H$ with the order $n$ we define on $G=aH$ an
$n$-ary operation by
$$
f(x_1,x_2,\ldots ,x_n)=ax_1x_2\ldots x_n.
$$
This operation is associative, because $a\in Z(A)$. Also, $G$ is
closed under this operation, since $o(a)=n$. So, $G$ is an $n$-ary
group. Any $A$-representation $\Lambda$ with $a\in\ker\Lambda$ is
also a $G$-representation.
\end{example}

\begin{example}\label{Ex-36}
The set $G=\mathbb{Z}_n$ with the ternary operation
$$
f(x,y,z)=x-y+z\,({\rm mod}\ n)
$$
is, by Theorem \ref{T-HG}, a ternary group. We want to classify
all representations of $G$.

\medskip

Let $\Lambda :G\rightarrow GL_m(\mathbb{C})$ be any
representation. Then we have
$$
\Lambda(f(x,y,z))=\Lambda(x)\Lambda(y)\Lambda(z),
$$
equivalently,
$$
\Lambda(x-y+z)=\Lambda(x)\Lambda(y)\Lambda(z).
$$
We have
$$
\Lambda(x+y)=\Lambda(x)\Lambda(0)\Lambda(y), \ \
\Lambda(x-y)=\Lambda(x)\Lambda(y)\Lambda(0).
$$
Suppose  $A=\Lambda(0)$. We have
$$
\Lambda(x+y)=\Lambda(x)A\Lambda(y).
$$
It is easy to see that $A^2=I$. Now, define
$\Lambda^{\prime}(x)=A\Lambda(x)$. Then
$$
\Lambda^{\prime}(x+y)=\Lambda^{\prime}(x)\Lambda^{\prime}(y),
$$
and so, $\Lambda^{\prime}$ is an ordinary representation of
$(\mathbb{Z}_n, +)$. Hence, every representation of the ternary
group $G$ is of the form $\Lambda(x)=A\Lambda^{\prime}(x)$, where
$A$ is an involution and $\Lambda^{\prime}$ is an ordinary
representation of $(\mathbb{Z}_n, +)$.

\medskip

Similarly, we can classify all representations of ternary groups
of the form $G=(A, f)$, where $A$ is an ordinary abelian group and
$$
f(x,y,z)=x-y+z.
$$
\end{example}

\begin{theorem}
{\sf (Maschke)} Let $G$ be a finite $n$-ary group. Then every
$G$-module is completely reducible.
\end{theorem}
\begin{proof} Let $(V, p)$ be a $G$-module and $W\leq_G V$.
Suppose $V=W\oplus X$, where $X$ is just a subspace. Let $\varphi:
V\rightarrow W$ be the corresponding projection. Define a  new map
$\theta:V\rightarrow V$ as
$$
\theta(v)=\frac{1}{|G|}\sum_{x\in G}\overline{x}.\varphi(x.v).
$$
It is easy to see that
\[
\theta(x.v)=x.p.\,\ldots\, .p.\theta(v)=x.\theta(v).
\]
So $\theta$ is a $G$-homomorphism and hence its kernel is a $G$-submodule. For all $w\in W$, we have $\theta(w)=w$ and so $\theta^2=\theta$.
Now, we have $V=W\oplus \ker \theta$.
\end{proof}

\begin{remark}
Any $G$-module $(V,p)$ is also an ordinary $Ret_p(G)$-module,
because
\[
(x\ast y).v=f(x,\stackrel{(n-2)}{p},y).v=x.p.\,\ldots\,
.p.y.v=x.y.v.
\]
\end{remark}

From now on, we  will  assume that $e\in G$ is an arbitrary fixed
element. For all $p\in G$, we have $Ret_e(G)\cong Ret_p(G)$ and
further the isomorphism is given by the following rule
$$
h(x)=f(\stackrel{(n-2)}{e},x,\overline{p}\,).
$$

By $\hat{G}$ we denote the binary group $Ret_e(G)$. If $(V, p)$ is
a $G$-module, then we can define a $\hat{G}$-module structure on
$V$ by $x\circ v=h(x).v$. So, we have
$$
x\circ v=f(\stackrel{(n-2)}{e},x,\overline{p}\,).v=e.\,\ldots\,
.e.x.\overline{p}.v.
$$
But, we have
$\overline{p}.v=\overline{p}.p.\,\ldots\,.p.v=f(\overline{p}\,,\stackrel{(n-1)}{p}).v=p.v=v$.
Hence
$$
x\circ v=\underbrace{e.\,\ldots\,.e}_{n-2}.x.v.
$$
Now, every $G$-module is also a $\hat{G}$-module, but the converse is not true in general. During
this article, we will give some necessary and sufficient conditions for a $\hat{G}$-module to be also a
$G$-module. The next proposition is the first condition of this type.

\begin{proposition}\label{P-cond}
Let $V$ be a $\hat{G}$-module. Then $V$ is a $G$-module iff
$$
\forall x_2,\ldots,x_{n-1}\in G\ \forall v:\
f(\overline{e},x_2^{n-1},\overline{e}\,).v=x_2.x_3.\,\ldots\,.x_{n-1}.v.
$$
\end{proposition}
\begin{proof} We have
\begin{eqnarray*}
f(x_1^n)&=&f(f(x_1,\stackrel{(n-2)}{e},\overline{e}\,),x_2^n)\\
        &=&f(x_1,\stackrel{(n-2)}{e},f(\overline{e},x_2^n))\\
        &=&x_1\ast f(\overline{e},x_2^n)\\
        &=&x_1\ast f(\overline{e},x_2^{n-1},f(\overline{e},\stackrel{(n-2)}{e},x_n))\\
        &=&x_1\ast f(\overline{e},x_2^{n-1},\overline{e}\,)\ast x_n .
\end{eqnarray*}
So, the equality
$$
f(x_1^n).v=x_1.x_2.\,\ldots\,.x_{n-1}.x_n.v
$$
holds, iff
$$
f(\overline{e},x_2^{n-1},\overline{e}\,).v=x_2.x_3.\,\ldots\,.x_{n-1}.v
$$
for all $x_2,\ldots,x_{n-1}$ and $v$.
\end{proof}

\begin{remark}
Suppose that $V$ is a $G$-module in which the corresponding
representation is $\Lambda$. We know that $V$ is also a
$\hat{G}$-module. The corresponding representation of this last
module is
$$
\hat{\Lambda}(x)=\underbrace{\Lambda(e)\ldots\Lambda(e)}_{n-2}\Lambda(x).
$$
Because in $\hat{G}$, the identity element is $\overline{e}$, we
have
$$
\hat{\Lambda}(\overline{e})=id.
$$
So $\Lambda(e)^{n-2}\Lambda(\overline{e})=id$ and hence
$$
\Lambda(\overline{e})=\Lambda(e)^{2-n}.
$$
In the sequel, the corresponding character of $\hat{\Lambda}$, will be denoted by $\hat{\chi}$.
\end{remark}

\begin{proposition}
Suppose that $\Lambda$ is a representation of  $G$ with the
character $\chi$. Then $\chi$ is fixed on the conjugate classes of
$G$.
\end{proposition}
\begin{proof} Indeed, for every $b\in Cl_G(a)$ we have
$$
\Lambda(b)=\Lambda(f(x,a,\stackrel{(n-3)}{x},\overline{x}\,))=
\Lambda(x)\Lambda(a)\Lambda(x)^{n-3}\Lambda(\overline{x}),
$$
so
\begin{eqnarray*}
\chi(b)&=&Tr\ (\Lambda(x)\Lambda(a)\Lambda(x)^{n-3}\Lambda(\overline{x}))\\
         &=&Tr\ (\Lambda(x)\Lambda(a)\Lambda(e)^{n-2}\Lambda(\overline{e})\Lambda(x)^{n-3}\Lambda(\overline{x}))\\
         &=&Tr\ (\Lambda(a)\Lambda(e)^{n-2}\Lambda(\overline{e})\Lambda(x)^{n-3}\Lambda(\overline{x})\Lambda(x))\\
         &=&Tr\ (\Lambda(a)\Lambda(e)^{n-2}\Lambda(f(\overline{e}, \stackrel{(n-3)}{x}, \overline{x},x)))\\
         &=&Tr\ (\Lambda(a)\Lambda(e)^{n-2}\Lambda(\overline{e}))\\
         &=&Tr\ \Lambda(a)\\
         &=&\chi(a).
\end{eqnarray*}
This completes the proof.
\end{proof}

\begin{proposition}
Suppose that $\Lambda:(G,f)\rightarrow GL(V)$ is a representation
of the finite $n$-ary group $(G,f)$ with the corresponding
character $\chi$. Let
$$
\ker \chi=\{ x\in G: \chi(x)=\dim V\}.
$$
Then $\ker \chi=\ker \Lambda$.
\end{proposition}
\begin{proof} Let $\dim V=m$. It is clear that $\ker
\Lambda\subseteq\ker\chi$. Moreover, for each $x\in G$ of order
$k$ we have
$$
\Lambda(x)^{m^k}=\Lambda(x).
$$
Hence $\Lambda(x)$ is a root of the polynomial $T^{m^k-1}-1$. But,
this polynomial has distinct roots in $\mathbb{C}$, so
$\Lambda(x)$ can be diagonalized, i.e.,
$$
\Lambda(x)\sim diag(\varepsilon_1, \ldots, \varepsilon_m),
$$
where all $\varepsilon_i$ are roots of unity. Now, we have
$$
\chi(x)=\varepsilon_1+ \cdots+ \varepsilon_m.
$$
If $\chi(x)=m$, then $\varepsilon_i=1$ for all $i$. Hence
$\Lambda(x)=id$ and so $x\in \ker \Lambda$. This completes the
proof.
\end{proof}

In the next proposition, we obtain the explicit form of the
character $\hat{\chi}$.

\begin{proposition}
Let $\chi$ be a character of an $n$-ary group $(G,f)$. Then for
any $p\in\ker\chi$ we have
$$
\hat{\chi}(x)=\chi(f(\stackrel{(n-2)}{e},x,\overline{p}\,)).
$$
\end{proposition}
\begin{proof} We know that $\chi$ is a character of $Ret_p(G)$.
On the other hand there is an isomorphism
$$
h: Ret_e(G)\rightarrow Ret_p(G) ,
$$
where $h(x)=f(\stackrel{(n-2)}{e},x,\overline{p}\,)$. So, the
composite map $\chi \circ h$ is a character of $Ret_e(G)$. Let
$\Lambda$ be the corresponding representation of $\chi$. Now, we
have
\begin{eqnarray*}
\chi(h(x))&=&Tr\ (\Lambda(e)^{n-2}\Lambda(x)\Lambda(\overline{p}))\\
                &=&Tr\ (\Lambda(e)^{n-2}\Lambda(x))\\
                &=&Tr\ \hat{\Lambda}(x).
\end{eqnarray*}
Hence
$\hat{\chi}(x)=\chi(f(\stackrel{(n-2)}{e},x,\overline{p}\,))$.
\end{proof}

\begin{remark}
Now, for any irreducible character $\chi$ of an $n$-ary group
$(G,f)$, we have an ordinary irreducible character $\hat{\chi}$ of
the binary group $\hat{G}=Ret_e(G)$. So, we obtain the following
orthogonality relation for the irreducible characters of $G$:
$$
\frac{1}{|G|}\sum_{x\in
G}\chi_1(f(\stackrel{(n-2)}{e},x,\overline{p}_1))\overline{\chi_2(f(\stackrel{(n-2)}{e},x,\overline{p}_2))}=\delta_{\hat{\chi}_1,
\hat{\chi}_2},
$$
where $p_1\in \ker \chi_1$ and $p_2\in \ker \chi_2$ are arbitrary
elements.
\end{remark}

\begin{proposition}\label{P-rep}
If a representation $\Gamma: Ret_e(G,f)\rightarrow GL(V)$ is also
a representation of the $n$-ary group $(G,f)$, then
$$
\Gamma(\overline{x})=\Gamma(x)^{2-n}
$$
for every $x\in G$.
\end{proposition}
\begin{proof} Indeed, $f(\stackrel{(n-1)}{x},\overline{x})=x$ implies
$\Gamma(x)^{n-1}\Gamma(\overline{x})=\Gamma(x)$, which gives
$\Gamma(\overline{x})=\Gamma(x)^{2-n}$.
\end{proof}

\begin{corollary}\label{C-ter}
Let $(G,f)$ be a ternary group. Then a representation $\Gamma:
Ret_e(G,f)\rightarrow GL(V)$ is also a representation of $(G,f)$
iff
$$
\Gamma(\overline{x})=\Gamma(x)^{-1}
$$
for every $x\in G$.
\end{corollary}
\begin{proof}
From Proposition \ref{P-cond} it follows that
$\Gamma:Ret_e(G,f)\rightarrow GL(V)$ is a representation of a
ternary group $(G,f)$ iff it satisfies the identity
$$
\Gamma(f(\overline{e},x,\overline{e}\,))=\Gamma(x).
$$
If $\Gamma(\overline{x})=\Gamma(x)^{-1}$ holds for all $x\in G$,
then, in view of \eqref{inverse}, for all $x\in G$ we have
\[
\Gamma(f(\overline{e},x,\overline{e}\,))=\Gamma(f(\overline{e},\overline{\overline{x}},\overline{e}\,))=
\Gamma(\overline{x}^{\,-1})=\Gamma(\overline{x})^{-1}=\Gamma(x).
\]
Hence $\Gamma$ is a representation of $(G,f)$.

The converse statement is a consequence of Proposition
\ref{P-rep}.
\end{proof}

\begin{remark}
We can use the above proposition to obtain some deeper results in
the case when $G$ has a central element. Note that, according to
\cite{DM'84}, an $n$-ary group $(G,f)$ has a central element iff
it is $b$-derived from a binary group $(G,\cdot)$ and $b\in
Z(G,\cdot)$. Obviously, in this case $Z(G,f)=Z(G,\cdot)$.
\end{remark}

\begin{proposition}
Let $e$ be a central element of an $n$-ary group
$(G,f)=der_b(G,\cdot)$. Then a representation
$\Gamma:Ret_e(G)\rightarrow GL(V)$ is a representation of
$(G,f)$ iff
$$\Gamma(x_2x_3\ldots
x_ne^{2-n})=\Gamma(x_2)\Gamma(x_3)\ldots\Gamma(x_n)
$$
for all $x_2,\ldots,x_n\in G$.
\end{proposition}
\begin{proof}
Since $(G,f)=der_b(G,\cdot)$ the binary operation in $Ret_e(G,f)$
has the form
$$
x\ast y=f(x,\stackrel{(n-2)}{e},y)=xye^{n-2}b.
$$
For a representation $\Gamma$ of $Ret_e(G,f)$, we have
\begin{equation}\label{dag}
\Gamma(x\ast y)=\Gamma(x)\Gamma(y) .
\end{equation}
Now, for $\Gamma$ to be a representation of $(G,f)$, it is
necessary and sufficient that
$$
\Gamma(f(x_1^n))=\Gamma(x_1x_2\ldots
x_nb)=\Gamma(x_1)\Gamma(x_2)\ldots\Gamma(x_n).
$$
If we replace in \eqref{dag}, $y$ by $x_2\ldots x_ne^{2-n}$, we
obtain
$$
\Gamma(x_1x_2\ldots x_nb)=\Gamma(x_1)\Gamma(x_2\ldots x_ne^{2-n}).
$$
So $\Gamma$ is a representation of $(G,f)$, iff
$$
\Gamma(x_2x_3\ldots x_ne^{2-n})=\Gamma(x_2)\Gamma(x_3)\ldots
\Gamma(x_n)
$$
for all $x_2,\ldots,x_n\in G$.
\end{proof}

In an $n$-ary group $(G,f)=der_b(G,\cdot)$ we have
$\overline{x}=x^{2-n}b^{-1}$. Hence, comparing the above result
with Proposition \ref{P-rep} we obtain

\begin{corollary}
Let $e$ be a central element of an $n$-ary group
$(G,f)=der_b(G,\cdot)$. If a representation $\Gamma:Ret_e(G)
\rightarrow GL(V)$ is a representation of $(G,f)$, then
$\Gamma(x^{2-n}b^{-1})=\Gamma(x)^{2-n}$ for every $x\in G$.
\end{corollary}

In the case of ternary groups, by Corollary \ref{C-ter}, we obtain
stronger result.

\begin{corollary}
Let $(G,f)=der_b(G,\cdot)$ be a ternary group. Then a
representation $\Gamma:Ret_e(G,f)\rightarrow GL(V)$ is also a
representation of $(G,f)$, iff $\Gamma((bx)^{-1})=\Gamma(x)^{-1}$
for every $x\in G$.
\end{corollary}

\begin{proposition}
Let $e$ be a central element of an ternary group
$(G,f)=der_b(G,\cdot)$. Then a character $\chi$ of $Ret_e(G,f)$ is
a character of $(G,f)$ iff for all $x\in G$ we have
$\chi(\bar{x})=\overline{\chi(x)}$.
\end{proposition}
\begin{proof}
Let $\Gamma:Ret_e(G,f)\rightarrow GL(V)$ be a representation
corresponding to $\chi$. If $\chi$ is a character of $(G,f)$, then
$\Gamma$ is also a representation of $(G,f)$ and so $
\Gamma(\bar{x})=\Gamma(x)^{-1}.$ Hence we have
$\chi(\bar{x})=\overline{\chi(x)}$.

Conversely, if $\chi(\bar{x})=\overline{\chi(x)}$ holds for all
$x\in G$, then in particular $\overline{\chi(e)}=\chi(\bar{e})$.
Thus $\chi(e)=\chi(\bar{e})$ because $\chi(\overline{e})$ is real.
Now, for all $x\in G$, we have
$x\ast\bar{x}=f(x,e,\bar{x})=f(e,x,\bar{x})=e$, so
$\chi(x*\bar{x})=\chi(e)=\chi(\bar{e})$. Hence,
$$
x\ast\bar{x}\in \ker \chi=\ker \Gamma.
$$
This shows that $\Gamma(x^{-1})=\Gamma(\bar{x})$ and so $\Gamma$ is a representation of $G$.
Hence $\chi$ is also a character of $G$.
\end{proof}

\begin{proposition}
Let $e$ be a central element of a ternary group
$(G,f)=der_b(G,\cdot)$. If $\chi$ is a common character of $(G,f)$
and $Ret_e(G,f)$, then $\hat{\chi}=\chi$.
\end{proposition}
\begin{proof} We have $\chi(\bar{e})=\overline{\chi(e)}$, so $\chi(e)$ is real, and hence $\chi(e)=\chi(\bar{e})$. So
$e\in \ker \chi$. Now, suppose $p=e$. Then
\[
\hat{\chi}(x)=
\chi(f(e,x,\bar{p}))=\chi(f(e,x,\bar{e}))=\chi(f(x,e,\bar{e}))=\chi(x),
\]
which completes the proof.
\end{proof}

In the remaining part of this section, we try to answer the
problem: when $\hat{\Lambda_1}\sim \hat{\Lambda_2}$? We give an answer to
this question for $n$-ary groups with some central elements.

\begin{proposition}
For an $n$-ary group $(G,f)$ with a central element $e$ the
following assertions are true:
\begin{enumerate}
\item \ Let $(V, p)$ be a $G$-module and $h:V\rightarrow V$ be a
$\hat{G}$-homomorphism. Then $h$ is also a $G$-homomorphism.

\item \ Let $(V_1, p_1)$ and $(V_2,p_2)$ be two $G$-modules and
$h:V_1\rightarrow V_2$ be a $\hat{G}$-homomorphism. Then $h$ is a
$G$-homomorphism, iff $h(e.v)=e.h(v)$.

\item \ Let $(V_1, p_1)$ and $(V_2,p_2)$ be two $G$-modules and
$h:V_1\rightarrow V_2$ be a $\hat{G}$-homomorphism. Then $h$ is a
$G$-homomorphism, iff $p_1.h(v)=h(v)$ for every $v\in V_1$.

\item \ Let $(V_1,p_1)$ and $(V_2,p_2)$ be two $G$-modules and
$$
V_1\cong_{\hat{G}} V_2.
$$
Then $V_1\cong_G V_2$, iff for all $u\in V_2$, $p_1.u=u$.
\end{enumerate}
\end{proposition}
\begin{proof} $(1)$. In view of $x\ast y=f(x,\stackrel{(n-2)}{e},y)$, for a $G$-module $(V,p)$, we have
\begin{eqnarray*}
h(e.v)&=&h(f(\stackrel{(n-1)}{e},\overline{e}\,).v)\\
      &=&h(f(f(\stackrel{(n-1)}{e},\overline{e}\,),\stackrel{(n-1)}{p}).v)\\
      &=&h(f(f(e,\stackrel{(n-2)}{p},\overline{e}\,),\stackrel{(n-2)}{e},p).v)\\
      &=&h(f(e,\stackrel{(n-2)}{p},\overline{e})\circ v)\\
      &=&f(e,\stackrel{(n-2)}{p},\overline{e}\,)\circ h(v)\\
      &=&f(f(e,\stackrel{(n-2)}{p},\overline{e}\,),\stackrel{(n-2)}{e},p).h(v)\\
      &=&f(e,\stackrel{(n-2)}{p},f(\overline{e},\stackrel{(n-2)}{e},p)).h(v)\\
      &=&f(e,\stackrel{(n-1)}{p}).h(v)\\
      &=&e.p.\,\ldots\,.p.h(v)\\
      &=&e.h(v).
\end{eqnarray*}
Now for all $x\in G$, we have $h(x\circ v)=x\circ h(v)$, so
$$
h(\underbrace{e.\ \ldots\  .e}_{n-2}. x. v))=e.\ \ldots\  .e. x. h(v).
$$
Hence
$$
\underbrace{e.\,\ldots\,.e}_{n-2}.h(x.v)=e.\,\ldots\,.e.x.h(v).
$$
Since the map $u\mapsto e.u$ is bijection, we have $h(x.v)=x.h(v)$.\\
$(2)$. The proof of this part is just as the above.\\
$(3)$. Suppose $h$ is a $G$-homomorphism. Then
$p_1.h(v)=h(p_1.v)=h(v)$ for every $v\in V_1$.

Conversely, assume that for all $v\in V_1$ holds $p_1.h(v)=h(v)$.
Then

\begin{eqnarray*}
h(e.v)&=&h(f(\stackrel{(n-1)}{e},\overline{e}).\underbrace{p_1.\ \ldots\ .p_1}_{n-2}.v)\\
      &=&h(\underbrace{e.\ \ldots\ .e}_{n-1}.\overline{e}.\underbrace{p_1.\ \ldots\ .p_1}_{n-2}.v)\\
      &=&h(f(e,\stackrel{(n-2)}{p_1},\overline{e}).\underbrace{e.\ \ldots\ .e}_{n-2}.v)\\
      &=&h(f(e, \stackrel{(n-2)}{p_1},\overline{e})\circ v)\\
      &=&f(e, \stackrel{(n-2)}{p_1},\overline{e})\circ h(v)\\
      &=&f(e,\stackrel{(n-2)}{p_1},\overline{e}).\underbrace{e.\ \ldots\ .e}_{n-2}.h(v)\\
      &=&f(e,\stackrel{(n-2)}{p_1},\overline{e}).\underbrace{e.\ \ldots\ .e}_{n-2}.p_1.h(v)\\
      &=&f(f(e,\stackrel{(n-2)}{p_1},\overline{e}),\stackrel{(n-2)}{e},p_1).h(v)\\
      &=&f( e,\stackrel{(n-2)}{p_1}, f( \overline{e},\stackrel{(n-2)}{e},p_1)).h(v)\\
      &=&f(e,\stackrel{(n-1)}{p_1}).h(v)\\
      &=&e.h(v).
\end{eqnarray*}
$(4)$. Let $h:V_1\rightarrow V_2$ be a $G$-isomorphism. Then $h$
is also a $\hat{G}$-homomorphism, and hence $p_1.h=h$. Because $h$
is onto, we obtain $p_1.u=u$, for all $u\in V_2$.

Conversely, suppose $p_1.u=u$, for all $u\in V_2$. Let
$h:V_1\rightarrow V_2$ be a $\hat{G}$-isomorphism. Then $p_1.h=h$,
and so $h$ is a $G$-isomorphism.
\end{proof}

\begin{proposition}
Let $(G,f)$ be an $n$-ary group with a central element and let
$\Lambda_1, \Lambda_2:G\rightarrow GL(V)$ be two representations
of $(G,f)$, such that $\hat{\Lambda}_1\sim \hat{\Lambda}_2$. Then
$\Lambda_1\sim \Lambda_2$, iff $\,\ker \Lambda_1=\ker \Lambda_2$.
\end{proposition}
\begin{proof}
Let $p\in \ker \Lambda_1=\ker \Lambda_2$. We define two
$G$-modules $V_1$ and $V_2$, as follows: $V_1$ is the vector space
$V$ with the action $x.v=\Lambda_1(x)(v)$, $V_2$ is the vector
space $V$ with the action $x. v=\Lambda_2(x)(v)$. Then
$\hat{\Lambda}_1\sim \hat{\Lambda}_2$ implies
$$
V_1\cong_{\hat{G}}V_2,
$$
and $p. u=u$, for all $u\in V_2$. So, $V_1\cong_G V_2$. This
proves $\Lambda_1\sim \Lambda_2$.

Conversely, let $\Lambda_1\sim \Lambda_2$. Hence, we have
$V_1\cong_G V_2$. By the previous proposition, for $p\in \ker
\Lambda_1$ and $u\in V_2$, we have $p.u=u$. Thus
$\Lambda_2(p)=id$. Therefore, $\ker\Lambda_1=\ker\Lambda_2$.
\end{proof}

\begin{corollary}\label{cc}
Let $\Lambda_1$ and $\Lambda_2$ be two representations of an $n$-ary group
$(G,f)$ with a central element $e$. If
$\hat{\Lambda}_1\sim\hat{\Lambda}_2$, then $\Lambda_1\sim \Lambda_2$ iff
$\Lambda_1(\overline{e})\sim \Lambda_2(\overline{e})$.
\end{corollary}

\begin{proof}  By the above proposition, $\Lambda_1\sim \Lambda_2$, iff $\ker \Lambda_1=\ker \Lambda_2$. But, we have
$$
\ker \Lambda_1=\{ x\in G: \hat{\Lambda}_1(x)=\Lambda_1(e)^{n-2}\},
$$
$$
\ker \Lambda_2=\{ x\in G: \hat{\Lambda}_2(x)=\Lambda_2(e)^{n-2}\}.
$$
Hence $\Lambda_1\sim \Lambda_2$, iff $\Lambda_1(e)^{n-2}\sim \Lambda_2(e)^{n-2}$. But we have $\Lambda_1(e)^{n-2}=\Lambda_1(\overline{e})^{-1}$ and similarly for $\Lambda_2$. So $\Lambda_1\sim \Lambda_2$, iff $\Lambda_1(\overline{e})^{-1}\sim \Lambda_2(\overline{e})^{-1}$.

\end{proof}

\begin{remark}
In the last two propositions and Corollary \ref{cc} the assumption
that $e$ is a central element can be replaced by the assumption
that that an $n$-ary group $(G,f)$ is semiabelian.
\end{remark}

\section{Connection with the representations of the covering group}

According to Post's Coset Theorem (cf. \cite{Post} or
\cite{cover}) for any $n$-ary group $(G,f)$ there exists a binary
group $(G^{\ast},\cdot )$ and its normal subgroup $H$ such that
$\,G^{\ast}\diagup H\simeq\mathbb{Z}_{n-1}\,$ and $G\subseteq G^{\ast}$ and
\[
f(x_1^n)=x_1\cdot x_2\cdot x_3\cdot\ldots\cdot x_n
\]
for all $x_1,\ldots,x_n\in G$.

The group $(G^{\ast},\cdot)$ is called the {\it covering group}
for $(G,f)$. We know several methods of a construction of such
group. The smallest covering group has the form
$G_a^{\ast}=G\times\mathbb{Z}_{n-1}$, where
\[
\langle x,r\rangle\!\cdot\!\langle y,s\rangle = \langle f_{\ast}
(x,\stackrel{(r)}{a},y,\stackrel{(s)}{a},\overline{a},
\stackrel{(n-2-r\diamond s)}{a}),\,r\diamond s\rangle ,
\]
$r\diamond s=(r+s+1)({\rm mod}\,(n-1))$ and $a\in G$ an arbitrary
but fixed element. The symbol $f_{\ast}$ means that the operation
$f$ is used one or two times (depending on the value $s$ and $t$).
Clearly fixing various element $a$ of $G$, we obtain various groups but
all these groups are isomorphic (cf. \cite{cover}).

The element $(\overline{a},n-2)$ is the identity of the group
$(G_a^{\ast}, \cdot)$. The inverse element has the form
$$
\langle x,t\rangle^{-1}= \langle
f_{\ast}(\overline{a},\stackrel{(n-2-t)}{a},\overline{x},\stackrel{(n-3)}{x},\overline{a},\stackrel{(t+1)}{a}),
k\rangle,
$$
where $k=(n-3-t)({\rm mod}\,(n-1))$.

The set $G$ is identified with the subset $\{\langle x,0\rangle:
x\in G\}$. Every retract of $(G,f)$ is isomorphic to the normal
subgroup
$$
H=\{\langle x,n-2\rangle: x\in G\}.
$$

Suppose that $V$ is a $G^{\ast}_a$-module. Then for
$x_1,\ldots,x_n\in G$ we have
\begin{eqnarray*}
x_1.x_2.x_3.\,\ldots\,.x_n.v&=&\langle x_1,0\rangle .\langle x_2,0\rangle .\langle x_3,0\rangle .\,\ldots\,.
\langle x_n,0\rangle .v\\
&=&\langle f(x_1,x_2,\overline{a},\stackrel{(n-3)}{a}),1\rangle .\langle x_3,0\rangle .\,\ldots\,.\langle x_n,0\rangle .v\\
&=&\langle f(f(x_1^2,\overline{a},\stackrel{(n-3)}{a}),a,x_3,\overline{a},\stackrel{(n-4)}{a}),2\rangle .\,\ldots\,
.\langle x_n,0\rangle .v\\
&=&\langle f(x_1^2,f(\overline{a},\stackrel{(n-2)}{a},x_3),\overline{a},\stackrel{(n-4)}{a}),2\rangle .\,\ldots\,
.\langle x_n,0\rangle .v\\
&=&\langle f(x_1^3,\overline{a},\stackrel{(n-4)}{a}),2\rangle .\,\ldots\,.\langle x_n,0\rangle .v\\
&\vdots&\ \\
&=&\langle f(x_1^n),0\rangle .v\\
&=&f(x_1^n).v
\end{eqnarray*}
So, we obtain

\begin{proposition}
Let $(G^{\ast}_a,\cdot)$ be the covering group for an $n$-ary group
$(G,f)$. Then for a $G^{\ast}_a$-module $V$ to be a $G$-module it is
necessary and sufficient that
$$
\exists p\in G\ \forall v\in V:\ p.v=v.
$$
\end{proposition}

Hence, we proved

\begin{proposition}
Let $(G^{\ast}_a,\cdot)$ be the covering group for an $n$-ary
group $(G,f)$. A representation $\Gamma$ of $G^{\ast}_a$ is a
representation of $G$, iff  $\,\ker \Gamma \cap G\neq
\varnothing$. If $\Gamma$ is irreducible
$G^{\ast}$-representation, then it is also irreducible as a
representation of $G$.
\end{proposition}

Now, suppose $(V,p)$ is a $G$-module. For the covering group
$(G_p^{\ast},\cdot)$ of $(G,f)$ we can define an action of $G_p^{\ast}$ on
$V$ as
$$
\langle x,k\rangle .v=x.v.
$$
Then, it can be easily verified that $V$ is a $G_p^{\ast}$-module.
But, we know that $G_a^{\ast}\cong G_p^{\ast}$, so let $h:
G_a^{\ast}\rightarrow G_p^{\ast}$ be any isomorphism. For any
$x\in G_a^{\ast}$, define $x.v=h(x).v$. Hence $V$ becomes a
$G_a^{\ast}$-module. Further, if $W$ is a $G$-submodule of $G$,
then it is also a $G_p^{\ast}$-submodule and so a
$G_a^{\ast}$-submodule. Hence, we proved

\begin{theorem}
There is a bijection between the set of all irreducible
representations of $(G,f)$ and the set of all irreducible
representations of $G^{\ast}_a$ with kernels not disjoint from $G$.
\end{theorem}

\section{Normal subgroups in polyadic groups}
In this section, we show that  the representation
theory of  $n$-ary groups reduces to the representation theory of binary groups. For this we
introduce the concept of normal $n$-ary subgroup.

\begin{definition}
An $n$-ary subgroup $H$ of an $n$-ary group $(G,f)$ is called {\it
normal} if
$$
f(\stackrel{(n-3)}{a},\overline{a},h,a)\in H
$$
for all $h\in H$ and $a\in G$. A normal subgroup $H\ne G$
containing at least two elements is called {\it proper}. If $G$ has no any proper normal subgroup, then we say that it is {\em simple}. If $H=G$ is the only simple subgroup of $G$, then we say it is {\em strongly simple}.
\end{definition}

\begin{definition}
For any $n$-ary subgroup $H$ of an $n$-ary group $(G,f)$ we define
the relation $\sim_H$ on $G$, by
$$
a\sim_H b\ \Longleftrightarrow\ \exists x, y\in H:\
b=f(a,\stackrel{(n-2)}{x},y).
$$
\end{definition}

\begin{lemma}\label{L-equiv}
$ a\sim_H b\ \Longleftrightarrow\ \exists x_2,\ldots,x_n\in H:\
b=f(a,x_2^n)$.
\end{lemma}
\begin{proof}
Indeed, if $b=f(a,x_2^n)$ for some $x_2,\ldots,x_n\in H$, then, in
view of Theorem \ref{dor-th}, for every $x\in H$ we have
$$
b=f(a,x_2^n)=f(a,f(\stackrel{(n-2)}{x},\overline{x},x_2),x_3^n)=
f(a,\stackrel{(n-2)}{x},y),
$$
where $y=f(\overline{x},x_2^n)\in H$, so $a\sim_H b$. The converse
is obvious.
\end{proof}
Now it is easy to see that such defined relation is an equivalence on $G$.
The equivalence class of $G$, containing $a$ is denoted by $aH$
and is called the {\it left coset} of $H$ with the representative
$a$. By Lemma \ref{L-equiv} it has the form
$$
aH=\{f(a,\stackrel{(n-2)}{x},y): x, y\in
H\}=\{f(a,h_2^n):h_2,\ldots,h_n\in H\}.
$$

The $n$-ary group $(G,f)$ is partitioned by cosets of $H$.

\begin{proposition}
If $H$ is a finite $n$-ary subgroup of $(G,f)$, then for all $a\in
G$, we have $|aH|=|H|$.
\end{proposition}

\begin{proof} By Theorem \ref{T-HG}, for an $n$-ary group $(G,f)$ there is a binary group $(G,\cdot)$,
$\varphi\in Aut(G,\cdot)$ and an element $b\in G$ such that
$$
f(x_1^n)=x_1\cdot\varphi(x_2)\cdot\varphi^2(x_3)\,\ldots\cdot\varphi^{n-1}(x_n)\cdot
b,
$$
for all $x_1,\ldots,x_n\in G$. So, we have
$$
aH=\{
a\cdot\varphi(x_2)\cdot\varphi^2(x_3)\,\ldots\cdot\varphi^{n-1}(x_n)\cdot
b: x_2,\ldots,x_n\in H\}.
$$
But, clearly this set is in one-one correspondence with the set
$$
\{
\varphi(x_2)\cdot\varphi^2(x_3)\,\ldots\cdot\varphi^{n-1}(x_n)\cdot
b: x_2,\ldots,x_n\in H\},
$$
which does not depend on $a$. So, we have $|aH|=|H|$.
\end{proof}

On the set $G/H=\{aH:a\in G\}$ we introduce the operation
$$
f_H(a_1H,a_2H,\ldots,a_nH)=f(a_1^n)H.
$$

\begin{proposition}
If $H$ is a normal $n$-ary subgroup of $(G,f)$, then $(G/H,f_H)$
is an $n$-ary group derived from the group $Ret_H(G/H,f)$.
\end{proposition}
\begin{proof}
First we show that the operation $f_H$ is well-defined. For this
let $a_iH=b_iH$ for some $a_i,b_i\in G$, $i=1,2,\ldots,n$. Then
\[
b_1=f(a_1,x_2^n), \ \ \ b_2=f(a_2,y_2^n), \ \ \ldots , \ \
b_n=f(a_n,z_2^n)
\]
for some $x_i, y_i,\ldots,z_i\in H$

Now, using Theorem \ref{T-HG} we obtain
\begin{eqnarray*}
f(b_1^n)&=&f(f(a_1,x_2^n),f(a_2,y_2^n),\ldots,f(a_n,z_2^n))\\
        &=&f(f(a_1,x_2^{n-1},f(\stackrel{(n-2)}{a_2},\overline{a}_2,x_n)),f(a_2,y_2^n),\ldots,f(a_n,z_1^n))\\
        &=&f(f(a_1,x_2^{n-1},a_2),f(f(\stackrel{(n-3)}{a_2},\overline{a}_2,x_n,a_2),y_2^n),\ldots,f(a_n,z_n^n))\\
        &=&f(f(a_1,x_2^{n-1},a_2),f(w_n,y_2^n),\ldots,f(a_n,z_1^n))\\
        &=&f(f(a_1,x_2^{n-2},f(\stackrel{(n-2)}{a_2},\overline{a}_2,x_{n-1}),a_2),f(w_n,y_2^n),\ldots,f(a_n,z_2^n))\\
        &=&f(f(a_1,x_2^{n-2},a_2,f(\stackrel{(n-3)}{a_2},\overline{a}_2,x_{n-1},a_2)),f(w_n,y_2^n),\ldots,f(a_n,z_2^n))\\
        &=&f(f(a_1,x_2^{n-2},a_2,w_{n-1}),f(w_n,y_2^n),\ldots,f(a_n,z_2^n))\\
        &\vdots&\ \\
        &=&f(f(a_1,a_2,w_3^{n-1}),f(w_n,y_2^n),\ldots,f(a_n,z_2^n)),
\end{eqnarray*}
where $w_i=f(\stackrel{(n-3)}{a_2},\overline{a}_2,x_i,a_2)\in H$.

Repeating this procedure for $a_3,a_4$ and so on, we obtain
$$
f(b_1^n)=f(f(a_1^n),h_2^n).
$$
This means that the operation $f_H$ is well-defined.

It is easy to verify that $(G/H,f_H)$ is an $n$-ary group. Using
the above procedure it is not difficult to see that $H$ is the
identity of $G/H$. Hence an $n$-ary group $G/H$ is derived from
the group $Ret_H(G/H)$.
\end{proof}

Now, we return to the representations, again. Consider a
representation $\Lambda:(G,f)\rightarrow GL(V)$. It is easy to see that $\ker \Lambda$ is a normal subgroup of $G$. Let $H$ be a
normal $n$-ary subgroup of $(G,f)$ such that $H\subseteq \ker
\Lambda$. Then, there is a representation $\bar{\Lambda}:G/H
\rightarrow GL(V)$ such that
$$
\bar{\Lambda}(aH)=\Lambda(a).
$$
Conversely, from every representation of $G/H$, we obtain a representation of $G$. On the other hand, $G/H$ is of reduced type, and hence
its representations are the same as the ordinary representations of $Ret_H(G/H)$. So, we proved,

\begin{proposition}
There is a bijection between ordinary representations of $Ret_H(G/H)$ and the set of representations of $G$ with the property $H\subseteq \ker \Lambda$.
\end{proposition}

\begin{proposition}
A simple $n$-ary group which is not strongly simple is $b$-derived
from an abelian group or it is reducible to a non-abelian group.
\end{proposition}
\begin{proof} Suppose $H=\{ p\}$ is a normal $n$-ary subgroup of $(G,f)$. Then we have
$$
f(p,p,\ldots,p)=p,\ \ \overline{p}=p,\ \ \forall x\in G:
f(\stackrel{(n-3)}{x},\overline{x},p,x)=p.
$$
Hence
\begin{eqnarray*}
f(p,x_2^n)&=&f(f(\stackrel{(n-2)}{x_2},\overline{x}_2,p),x_2^n)\\
     &=&f(x_2,f(\stackrel{(n-3)}{x_2},\overline{x}_2,p,x_2),x_3^n)\\
     &=&f(x_2,p,x_3^n).
\end{eqnarray*}
This shows that $p$ is a central element and, according to
\cite{DM'84}, an $n$-ary group $(G,f)$ is $b$-derived from a
binary group $(G,\cdot)$. Hence, $Z(G,f)=Z(G,\cdot)$ is a normal
$n$-ary subgroup of $(G,f)$. But $G$ has no proper normal
subgroups, so there are two cases:
\begin{enumerate}
\item $Z(G,\cdot)=G$ and so $(G,f)$ is $b$-derived from an
abelian group,

\item $Z(G,\cdot)$ is singleton and hence $b=1$. In this case
$(G,f)$ is reducible to a non-abelian group $(G,\cdot)$.
\end{enumerate}
\end{proof}

\begin{remark}
To find representations of an $n$-ary group $(G,f)$, we have four
cases, as follow,
\begin{enumerate}
\item only $H=G$ is a normal subgroup of $(G,f)$, (in this
case $(G,f)$ has only trivial representation),

\item $(G,f)$ is $b$-derived from an abelian group,

\item $(G,f)=der(G,\cdot)$, (in this case representations of $(G,f)$ are the
same as the representations of $(G,\cdot)$),

\item $(G,f)$ has proper normal $n$-ary subgroups, (in this case, if we know the set of
normal $n$-ary subgroups of $(G,f)$, then we obtain all its
representations from representations of the groups $Ret_H(G/H)$).
\end{enumerate}
\end{remark}

Finally, summarizing results of this section, we have the following theorem:

\begin{theorem}
Representation theory of $n$-ary groups, reduces to the following
three problems,

$a)$ \ representations of $b$-derived ternary groups from abelian
groups,

$b)$ \ determining all normal $n$-ary subgroup,

$c)$ \ representation theory of ordinary groups.
\end{theorem}

%%%%%%%%%%%%%%%%%%%%%%%%%%%%%%%%%%%%%%%%%%%%%%%%%%%%%%%%%%%%%%%%%%%%%%%%%%%%%%%%%%%%%

\end{document}